\documentclass[preprint,12pt]{elsarticle}

\usepackage{watermark,graphicx,amssymb,amstext,amsmath,amsthm,enumerate,verbatim,tikz,nomencl}

\journal{Linear algebra and its applications}

\DeclareMathOperator{\sgn}{sgn}
\DeclareMathOperator{\supp}{supp}
\DeclareMathOperator{\id}{Id}
\DeclareMathOperator{\ad}{ad}

\newcommand\C{\mathbb{C}}

\newcommand\Z{\mathbb{Z}}
\newcommand\N{\mathbb{N}}

\newcommand\uq{U^{(q)}}
\newcommand\uf{U^{(f)}}
\DeclareMathOperator{\re}{Re}

\DeclareMathOperator{\Ind}{Ind_U^{U^{(u)}}}
\DeclareMathOperator{\Res}{Res_U^{U^{(u)}}}

\newcommand\ie{\textsl{i.e. }}

\newtheorem{mythm}{Theorem}
\newtheorem{mylem}[mythm]{Lemma}
\newtheorem{mycor}[mythm]{Corollary}
\newtheorem{myprop}[mythm]{Proposition}
\newdefinition{rmk}{Remark}
\newdefinition{mydef}{Definition}
\newproof{myproof}{Proof}
\newproof{myproofof}{Proof of Theorem \ref{mainresult}}

\begin{document}

\begin{frontmatter}



\title{Classification of simple weight modules with finite-dimensional weight spaces over the Schr\"odinger algebra}


\author{Brendan Dubsky}

\address{Department of Mathematics, Uppsala University, Box 480, 75106 Uppsala, Sweden}

\begin{abstract}
We classify simple weight modules with finite-dimensional weight spaces over the 
(centrally extended complex) Schr\"odinger algebra in $(1+1)$-dimensional space-time.
Our arguments use the description of lowest weight modules by Dobrev, Doebner and Mrugalla;
Mathieu's twisting functors and results of Wu and Zhu on
dimensions of weight spaces in dense modules.
\end{abstract}


\end{frontmatter}


\section{Introduction and preliminaries}
\label{s0}
The Schr\"odinger Lie group is the group of symmetries of the free particle Schr\"odinger equation. The (centrally extended) Lie algebra $\mathcal{S}$ of this group in the case of $(1+1)$-dimensional space-time is called the {\em Schr\"odinger algebra}, see \cite{Do97,Wu13}. The algebra $\mathcal{S}$ has basis $\{ f,q,h,z,p,e\}$ and 
the Lie bracket is given as follows:
\begin{equation}
\label{commrelations}
\begin{array}{lll}
\left[h,e\right]=2e, &  \left[h,p\right]=p,& \left[h,f\right]=-2f,\\ 
\left[e,q\right]=p, &  \left[e,p\right]=0, & \left[e,f\right]=h, \\
\left[p,f\right]=-q, &\left[f,q\right]=0, &\left[h,q\right]=-q,\\
\left[p,q\right]=z, &\left[z,\mathcal{S}\right]=0.
\end{array}
\end{equation}
From this we see that $e$, $f$ and $h$ generate an $\mathfrak{sl}_2$-subalgebra  of $\mathcal{S}$.
\par
In this paper we study so-called \emph{weight modules} over $\mathcal{S}$, that is $\mathcal{S}$-modules 
which are diagonalizable over the \emph{Cartan subalgebra} $\mathfrak{h}$ of $\mathcal{S}$ spanned by $h$ and $z$. 
Put differently, these are modules which may be written as a sum of simultaneous eigenspaces for $h$ and $z$, 
so-called \emph{weight spaces}. The goal of this paper is to classify, up to isomorphism, 
all simple weight $\mathcal{S}$-modules with finite-dimensional weight spaces. 
\par
As usual, $\mathcal{S}$-modules are the same as modules over the universal enveloping algebra $U$ of $\mathcal{S}$. 
By the Poincar\'e-Birkhoff-Witt theorem (see, e.g., \cite[p. 156-160]{Ja79}), $U$ is a unital and associative algebra 
with basis 
\begin{displaymath}
\{q^{i_1}f^{i_2}p^{i_3}e^{i_4}h^{i_5}z^{i_6}\}_{i_1,\dots,i_6\in\N},
\end{displaymath}
where $\N$ denotes the set of nonnegative integers. 
\par
Schur's lemma  says that the central element $z$ acts as a scalar on each simple $\mathcal{S}$-module $V$,
which is called the {\em central change} of $V$. Hence the weight spaces of a simple module are, in fact, 
precisely the eigenspaces of $h$. For a module $V$ on which $z$ acts as a scalar we denote by $\supp(V)$ 
the set of $h$-eigenvalues on $M$ and we will call these eigenvalues \emph{weights}. All $h$-eigenvectors 
will be called \emph{weight vectors}. For a weight $\lambda$ we denote by $V_{\lambda}$ the corresponding
weight space.
\par
The weights of a weight module $V$ are weakly ordered as follows: $\lambda_1\le \lambda_2$ if and only if $\lambda_2-\lambda_1$ is a nonnegative integer. If there is a maximal element, $\lambda$, in $\supp(V)$, then $\lambda$ is called a \emph{highest weight} of $V$, and any $v\in V_{\lambda}$ is called a \emph{highest weight vector}. If $V$ is generated by a highest weight vector, then it is called a \emph{highest weight module}. Lowest weights, lowest weight vectors and lowest weight modules are defined analogously. We have the following standard fact
(cf. \cite[Lemma~1.15]{Ma10}).
\par
\begin{mylem}
\label{weightspactions}
Let $V$ be a weight $U$-module with $\lambda\in\supp(V)$. Then the following hold:
\begin{displaymath}
\begin{array}{lll}
f:V_{\lambda}\rightarrow V_{\lambda-2},&
q:V_{\lambda}\rightarrow V_{\lambda-1},&
h:V_{\lambda}\rightarrow V_{\lambda},\\
z:V_{\lambda}\rightarrow V_{\lambda},&
p:V_{\lambda}\rightarrow V_{\lambda+1},&
e:V_{\lambda}\rightarrow V_{\lambda+2}.
\end{array}
\end{displaymath}
\end{mylem}
\par
Those $\mathcal{S}$-modules on which $p$ and $q$ (and therefore also $z$) act as $0$ are $\mathfrak{sl}_2$-modules, and vice versa, so a classification of these modules is well-known, see, e.g., \cite[p. 72]{Ma10}. 
\par
In the next section we recall known partial results from \cite{Do97,Wu13} on classification of simple weight $U$-modules with finite-dimensional weight spaces. In Section~\ref{s3} we formulate our main result, Theorem~\ref{mainresult}.
In Section~\ref{s35} we discuss our main tool, namely Mathieu's twisting functors from \cite{Ma00}. Finally, Theorem~\ref{mainresult} is proved in Section~\ref{s4}.

\section{Known partial results}\label{s2}

A classification and explicit description of all simple lowest weight $\mathcal{S}$-modules (on which either $p$ or $q$, or both, act nonzero) was given by Dobrev et al. in \cite{Do97}. Classification and explicit description of 
simple highest weight $\mathcal{S}$-modules follows easily. As usual, simple highest weight modules
are unique simple quotients of Verma modules, that is $\mathcal{S}$-modules induced from one-dimensional modules over the subalgebra spanned by $f$, $q$, $h$ and $z$, on which both $f$ and $q$ act as $0$. Dobrev et al.
explicitly compute maximal submodules in Verma modules. Here we recall their result in a slightly different
version which is necessary for our arguments. In the following proposition, as well as the rest of this paper, we for a set, $X$, we denote by $X^*$ the set $X\backslash\{0\}$.
\begin{myprop}
\label{highestclass}
For $\lambda\in\C\backslash (-\frac{1}{2}+\N)$ and $c\in\C^*$, let $M(\lambda,c)$ be the $U$-module 
with basis $\{ v_{i,j}\}_{i,j\in\mathbb{N}}$ on which the action of $U$ is given by
\begin{eqnarray*}
qv_{i,j}&=&v_{i+1,j};\\
fv_{i,j}&=&v_{i,j+1};\\
zv_{i,j}&=&cv_{i,j};\\
hv_{i,j}&=&(\lambda-i-2j)v_{i,j};\\
pv_{i,j}&=&-jv_{i+1,j-1}+civ_{i-1,j};\\
ev_{i,j}&=&j(\lambda+1-i-j)v_{i,j-1}+\frac{1}{2}ci(i-1)v_{i-2,j}.
\end{eqnarray*}
For $\lambda\in -\frac{1}{2}+\N$ and $c\in\C^*$, denote by  $N(\lambda,c)$ the $U$-module 
of $M(\lambda,c)$ with basis $\{ v_{i,j}\}_{i,j\in\mathbb{N}, j\le\lambda+\frac{1}{2}}$ on which the action of $U$ is given by
\begin{equation}
\label{highestclassq}
\begin{array}{rcl}
qv_{i,j}&=&v_{i+1,j};\\
fv_{i,j}&=&
\begin{cases}
v_{i,j+1},&\text{if $j<\lambda+\frac{1}{2}$};\\
-\sum_{s=0}^{\lambda+\frac{1}{2}}{\frac{1}{(2c)^{\lambda+\frac{3}{2}-s}}\binom{\lambda+\frac{3}{2}}{s}v_{i+2\lambda+3-2s,s}},&\text{if $j=\lambda+\frac{1}{2}$};
\end{cases}
\\
zv_{i,j}&=&cv_{i,j};\\
hv_{i,j}&=&(\lambda-i-2j)v_{i,j};\\
pv_{i,j}&=&-jv_{i+1,j-1}+civ_{i-1,j};\\
ev_{i,j}&=&j(\lambda+1-i-j)v_{i,j-1}+\frac{1}{2}ci(i-1)v_{i-2,j}.
\end{array}
\end{equation}
Then every simple highest weight $U$-module with finite-dimensional weight spaces on which either 
$p$ or $q$ or both act nonzero is isomorphic to precisely one module of the form $M(\lambda,c)$ or $N(\lambda,c)$. 
\end{myprop}
\begin{myproof}
This result follows easily from \cite[Theorem~1]{Do97}. Consider the Lie algebra $\mathcal{S}(1)$ defined in \cite{Do97}. It is readily checked that the map $\psi:\mathcal{S}(1)\rightarrow \mathcal{S}$ given by
\begin{displaymath}
\begin{aligned}
D\mapsto -h,\quad
P_t\mapsto -e,\quad
P_x\mapsto p,\quad
K\mapsto f,\quad
G\mapsto -q,\quad
m\mapsto -z
\end{aligned}
\end{displaymath}
extends to a Lie algebra isomorphism (here $m$ is identified with the scalar with which it acts on the modules in question). Now the statement follows by pushing forward \cite[Theorem~1]{Do97}  along $\psi$.
\qed
\end{myproof}

The simple weight $U$-modules with finite-dimensional weight spaces which now remain to be classified are those on which either $p$ or $q$ or both act nonzero and, furthermore, which have neither a highest nor a lowest weight. Let us denote the class of such modules by $\mathcal{N}$. 
\par
Theorem~\ref{chinesethm} below is due to Wu and Zhu, see \cite{Wu13}. It describes dimensions of weight space 
for modules in $\mathcal{N}$. The proof, though omitted here, requires the following lemma, which we will use.
\begin{mylem}
\label{localnil}
\begin{enumerate}[$($a$)$]
\item\label{localnil.1}
For $s\in \{p,q,e,f\}$ the adjoint action of $s$ on $U$ is locally nilpotent. 
\item\label{localnil.2}
Let $V$ be a $U$-module, $v\in V$, $u,s\in U$ and $n,m\in\N$. Assume that $\ad_s^nu=0$. If $s^{m}v=0$, then 
$s^{nm}uv=0$. 
\item\label{localnil.3}
Let $s\in \{p,q,e,f\}$ and $V$ be a simple $U$-module. If the kernel of $s$ on $V$ is nonzero, then
$s$ acts on $V$ locally nilpotently.
\end{enumerate}
\end{mylem}
\begin{myproof}
For $m\in\N$ let $U_m$ denote the linear span in $U$ 
of all standard monomials $q^{i_1}f^{i_2}p^{i_3}e^{i_4}h^{i_5}z^{i_6}$
such that $i_1+\dots+i_6\leq m$. Then  $U_m$ is finite-dimensional and stable under the adjoint action of
$s$. Moreover, $U_m$ is also stable under the adjoint action of $h$. At the same time, the adjoint action of
$h$ on $U_m$ is certainly diagonalizable (as all monomials are eigenvectors for this action). 
Since $U_m$ is finite-dimensional, the adjoint action of $h$ on $U_m$ has only finitely
many eigenvalues. Now claim \eqref{localnil.1} follows from Lemma~\ref{weightspactions}.

To prove claim \eqref{localnil.2}, let $X$ denote the linear subspace of $U$ spanned by all elements
of the form $s^ius^j$, $i,j\in\N$. The identity $\ad_s^nu=0$ can be written as $s^nu=u's$ 
for some $u'\in X$ which implies
that for any $x\in X$ there is $x'\in X$ such that $s^nx=x's$. By induction, it follows that 
for any $x\in X$ there is $x'\in X$ such that $s^{nm}x=x's^m$. Therefore there is $u'\in X$ such that 
$s^{nm}uv=u's^m v=0$, proving claim \eqref{localnil.2}.

As a simple module is generated by any nonzero element, claim \eqref{localnil.3} follows from claims \eqref{localnil.1} and \eqref{localnil.2}.
\qed
\end{myproof}
\begin{mythm}
\label{chinesethm}
Let $L\in\mathcal{N}$ and $\lambda\in\supp(L)$. Then $\supp(L)=\lambda+\mathbb{Z}$ and $\dim(L_{\lambda+i})=\dim(L_{\lambda+j})$ for all $i,j\in\mathbb{Z}$. 
\end{mythm}

\section{Main result}
\label{s3}

The following theorem is the main result of this paper.
\begin{mythm}
\label{mainresult}
For $\lambda\in -\frac{1}{2}+\N$, $c\in\C^*$ and $x\in \C$ with $x\ne 0$ and $0\le\re(x)<1$, let ${B}^{(q)}_x(N(\lambda,c))$ be the 
vector space with basis $\{ v_{i,j}\}_{i,j\in\mathbb{Z}, 0\le j\le\lambda+\frac{1}{2}}$. Then,
setting
\begin{equation}
\label{qaction}
\begin{array}{rcl}
qv_{i,j}&=&v_{i+1,j};\\
fv_{i,j}&=&
\begin{cases}
v_{i,j+1},&\text{if $j<\lambda+\frac{1}{2}$};\\
-\sum_{s=0}^{\lambda+\frac{1}{2}}{\frac{1}{(2c)^{\lambda+\frac{3}{2}-s}}\binom{\lambda+\frac{3}{2}}{s}v_{i+2\lambda+3-2s,s}},&\text{if $j=\lambda+\frac{1}{2}$};
\end{cases}
\\
zv_{i,j}&=&cv_{i,j};\\
hv_{i,j}&=&(\lambda-(i+x)-2j)v_{i,j};\\
pv_{i,j}&=&-jv_{i+1,j-1}+c(i+x)v_{i-1,j};\\
ev_{i,j}&=&j(\lambda+1-(i+x)-j)v_{i,j-1}+\frac{1}{2}c(i+x)((i+x)-1)v_{i-2,j};
\end{array}
\end{equation}
defines on ${B}^{(q)}_x(N(\lambda,c))$ the structure of a $U$-module. Moreover,
every module in $\mathcal{N}$ is isomorphic to precisely one module of this form.
\end{mythm}

The idea of the proof is the following: Consider the localization $\uf$ of $U$ with respect to $f$, which is equipped with a family of Mathieu's twisting functors as in \cite{Ma00}. We will show that every module in $\mathcal{N}$ can be ``twisted'' into a module which has a highest weight subquotient. Reversing the procedure we get that every module in $\mathcal{N}$ is (up to isomorphism) a twisted highest weight module. At this point, however, it is not clear how to handle possible redundancy issues. We will show that on all module in $\mathcal{N}$ the action of $z$ is nonzero. This information will allow us to repeat the entire procedure, with the difference that we now localize with respect to $q$ instead. This time redundancy is easily dealt with, so that the classification can be completed.
\par
We may now state the promised classification.
\begin{mycor}
Every simple weight $U$-module with finite-dimensional weight spaces is either:
\begin{itemize}
\item
A dense $\mathfrak{sl}_2$-module, see \cite[p. 72]{Ma10}.
\item
A highest weight $U$-module, see Proposition \ref{highestclass}.
\item
A lowest weight $U$-module, see \cite[Theorem~1]{Do97}.
\item
Of the form ${B}^{(q)}_x(N(\lambda,c))$, as defined in Theorem \ref{mainresult}.
\end{itemize}
\end{mycor}

\section{Mathieu's twisting functors}
\label{s35}

In this section we give all details on Mathieu's twisting functors from \cite{Ma00}, our main technical tool.

Let $u\in \{ q,f\}$. Then, by Lemma~\ref{localnil}, $\ad_u$ is locally nilpotent on $U$. Now by Lemma 4.2 
in \cite{Ma00}, $q$ and $f$ satisfy Ore's localizability conditions and hence we can consider the corresponding
localization $U^{(u)}$. We have $U\subset U^{(u)}$ is a ring extension in which $u$ is invertible, moreover, 
every element of $U^{(u)}$ can be written on the form $u^{-n}s$, for some $s\in U$ and $n\in\N$. For
$U^{(u)}$ we have the following analogue to the Poincar\'e-Birkhoff-Witt theorem.
\begin{myprop}
\label{pbwu}
The set $\{q^{i_1}f^{i_2}p^{i_3}e^{i_4}h^{i_5}z^{i_6}\}_{i_1\in\Z \text{ and } i_2\dots,i_6\in\N}$ is a basis for $\uq$, and $\{q^{i_1}f^{i_2}p^{i_3}e^{i_4}h^{i_5}z^{i_6}\}_{i_2\in\Z \text{ and } i_1,i_3,\dots,i_6\in\N}$ is a basis for $\uf$. 
\end{myprop}
\begin{proof}
Recall that, by the Poincar\'e-Birkhoff-Witt theorem, 
\begin{equation*}
\{q^{i_1}f^{i_2}p^{i_3}e^{i_4}h^{i_5}z^{i_6}\}_{i_1,\dots,i_6\in\N}
\end{equation*}
is a basis for $U$. By the definition of localization, 
\begin{equation*}
\{q^{i_1}f^{i_2}p^{i_3}e^{i_4}h^{i_5}z^{i_6}\}_{i_1\in\Z \text{ and } i_2\dots,i_6\in\N}
\end{equation*}
spans $\uq$. Next, assume that we have linearly dependent different elements 
\begin{equation*}
u_1,\dots,u_n\in\{q^{i_1}f^{i_2}p^{i_3}e^{i_4}h^{i_5}z^{i_6}\}_{i_1\in\Z \text{ and } i_2\dots,i_6\in\N}
\end{equation*}
so that $a_1u_1+\dots+a_nu_n=0$, for some $0\ne a_1,\dots, a_n\in\C$. Then for some $m\in\N$, the elements 
$q^mu_1,\dots,q^mu_n$ are different basis elements of $U$ and $a_1q^mu_1+\dots+a_nq^mu_n=0$, a contradiction.
\par
Since $q$ and $f$ commute, a similar argument works for $\uf$ as well.
\end{proof}

\begin{myprop}
\label{auto}
For $x\in\C$ the assignment
\begin{equation}
\label{autoformulae}
\begin{array}{lcllcl}
{\Theta}^{(q)}_x(q^{\pm 1})&=&q^{\pm 1},& {\Theta}^{(q)}_x(f)&=&f,\\
{\Theta}^{(q)}_x(z)&=&z,&{\Theta}^{(q)}_x(h)&=&h-x,\\ 
{\Theta}^{(q)}_x(p)&=&p+xq^{-1}z,\,\,&{\Theta}^{(q)}_x(e)&=&e+xq^{-1}p+\frac{1}{2}x(x-1)q^{-2}z
\end{array}
\end{equation}
extends uniquely to an automorphism ${\Theta}^{(q)}_x:\uq\rightarrow\uq$ and the assignment
\begin{equation}
\label{autoformulae2}
\begin{array}{lcllcl}
{\Theta}^{(f)}_x(q)&=&q,&{\Theta}^{(f)}_x(f^{\pm 1})&=&f^{\pm 1},\\
{\Theta}^{(f)}_x(z)&=&z,&{\Theta}^{(f)}_x(h)&=&h-2x,\\
{\Theta}^{(f)}_x(p)&=&p-xqf^{-1},\,\,\,&{\Theta}^{(f)}_x(e)&=&e+x(h-1-x)f^{-1}
\end{array}
\end{equation}
extends uniquely to an automorphism ${\Theta}^{(f)}_x:\uf\rightarrow\uf$.
\end{myprop}
\begin{myproof}
The uniqueness is clear as $\uq$ is generated by  $\{q^{\pm 1}, f, z, h, p, e\}$ and
$\uf$ is generated by $\{ q,f^{\pm 1}, z, h, p, e\}$. Hence we prove existence.
\par
Let $u\in\{q,f\}$ and assume that $x\in\mathbb{N}$. We claim that in this case 
Formulae~\ref{autoformulae} and~\ref{autoformulae2} correspond to restriction (to generators) of
the conjugation automorphism $a\mapsto u^{-x}au^{x}$ of $U^{(u)}$. To prove this we proceed by induction on
$x$. The base $x=0$ is immediate. Let us check the induction step from $x=k$ to $x=k+1$: For $u=q$ and $s=h$ we have
\begin{multline*}
q^{-1}\Theta^{(q)}_k(h)q
=q^{-1}(h-k)q
=q^{-1}(hq-kq)=\\
=q^{-1}(qh-q-kq)
=h-(k+1)
=\Theta^{(q)}_{k+1}(h).
\end{multline*}
For $u=q$ and $s=p$ we have 
\begin{multline*}
q^{-1}\Theta^{(q)}_k(p)q
=q^{-1}(p+kq^{-1}z)q
=q^{-1}pq+kq^{-1}z
=q^{-1}(qp+z)+kq^{-1}z=\\
=p+q^{-1}z+kq^{-1}z
=p+(k+1)q^{-1}z
=\Theta^{(q)}_{k+1}(p).
\end{multline*}
For $u=q$ and $s=e$ we have 
\begin{multline*}
q^{-1}\Theta^{(q)}_k(e)q
=q^{-1}(e+kq^{-1}p+\frac{1}{2}k(k-1)q^{-2}z)q=\\
=q^{-1}(eq+kq^{-1}pq)+\frac{1}{2}k(k-1)q^{-1}z=\\
=q^{-1}(qe+p+kq^{-1}(qp+z))+\frac{1}{2}k(k-1)q^{-2}z=\\
=e+q^{-1}p+kq^{-1}p+kq^{-2}z+\frac{1}{2}k(k-1)q^{-2}z=\\
=e+(k+1)q^{-1}p+\frac{1}{2}(k+1)(k+1-1)q^{-2}z
=\Theta^{(q)}_{k+1}(e).
\end{multline*}
\par
All other cases for $u=q$ are obvious. For $u=f$ and  $s=h$ we have
\begin{multline*}
f^{-1}\Theta^{(f)}_k(h)f
=f^{-1}(h-2k)f
=f^{-1}(hf-2kf)=\\
=f^{-1}(fh-2f-2kf)
=h-2(k+1)
=\Theta^{(f)}_{k+1}(h).
\end{multline*}
For $u=f$ and $s=p$ we have
\begin{multline*}
f^{-1}\Theta^{(f)}_k(p)f
=f^{-1}(p-kqf^{-1})f
=f^{-1}pf-kf^{-1}qf^{-1}f=\\
=f^{-1}(fp-q)-kf^{-1}q
=p-(k+1)qf^{-1}
=\Theta^{(f)}_{k+1}(p).
\end{multline*}
For $u=f$ and  $s=e$ we have
\begin{multline*}
f^{-1}\Theta^{(f)}_k(e)f
=f^{-1}(e+k(h-1-k)f^{-1})f
=f^{-1}(ef+k(h-1-k))=\\
=f^{-1}(fe+h+k(h-1-k))
=e+f^{-1}(k+1)(h-k)
=e+f^{-1}(k+1)(h-k)ff^{-1}=\\
=e+f^{-1}f(k+1)(h-2-k)f^{-1}
=e+(k+1)(h-1-(k+1))f^{-1}
=\Theta^{(f)}_{k+1}(e).
\end{multline*}
\par
To show that $\Theta^{(u)}_x$, given by extending Formulae~\ref{autoformulae} and~\ref{autoformulae2} linearly and multiplicatively, defines an endomorphism of  $U^{(u)}$ for general $x\in\mathbb{C}$, it suffices to show that each ${\Theta}^{(u)}_x$ preserves Lie brackets of the generators of $U^{(u)}$, \ie that ${\Theta}^{(u)}_x([s_1,s_2])={\Theta}^{(u)}_x(s_1){\Theta}^{(u)}_x(s_2)-\Theta^{(u)}_x(s_2)\Theta^{(u)}_x(s_1)$ for all $s_1,s_2\in\{ q, f, z, h, p, e, u^{-1}\}$. Let $s_1,s_2\in\{ q, f, z, h, p, e, u^{-1}\}$.
Using Formulae~\ref{autoformulae} and~\ref{autoformulae2} we can write 
\begin{equation}
Q_{s_1,s_2}:={\Theta}^{(u)}_x([s_1,s_2])-({\Theta}^{(u)}_x(s_1){\Theta}^{(u)}_x(s_2)-\Theta^{(u)}_x(s_1)\Theta^{(u)}_x(s_2))
\end{equation}
in the form $\sum_{i\in F}{g_i(x)b_i}$ where $F$ is a finite subset of the standard basis of $U^{(u)}$
and each $g_i(x)$ is a polynomial in $x$. From the above, we have $g_i(x)=0$ for all $x\in\N$. Hence
all $g_i=0$ and $Q_{s_1,s_2}=0$. 
\par
Finally, from the next proposition we have that $\Theta^{(u)}_x\circ\Theta^{(u)}_{-x}=\Theta^{(u)}_{0}=\id$, 
implying that $\Theta^{(u)}_x$ is invertible and thus an automorphism.
\qed
\end{myproof}
\begin{myprop}
\label{thetaadditive}
For all $x,y\in\C$ and $u\in\{q,f\}$ we have $\Theta^{(u)}_x\circ\Theta^{(u)}_y=\Theta^{(u)}_{x+y}$.
\end{myprop}
\begin{myproof}
This is a consequence of \eqref{autoformulae} and \eqref{autoformulae2} applied to the generators of $U^{(u)}$. The only non-immediate case is to check that $\Theta^{(u)}_x\circ\Theta^{(u)}_y(e)=\Theta^{(u)}_{x+y}(e)$.
This is done by the following direct calculation for $u=q$ and $u=f$, respectively:
\begin{equation*}
\begin{aligned}
\Theta^{(q)}_{x}\circ\Theta^{(q)}_y(e)&=\Theta^{(q)}_x(e+yq^{-1}p+\frac{1}{2}y(y-1)q^{-2}z)\\
&=e+xq^{-1}p+\frac{1}{2}x(x-1)q^{-2}z+yq^{-1}(p+xq^{-1}z)+\frac{1}{2}y(y-1)q^{-2}z\\
&=e+(x+y)q^{-1}p+\frac{1}{2}(x+y)(x+y-1)q^{-2}z\\
&=\Theta^{(q)}_{x+y}(e),
\end{aligned}
\end{equation*}
and
\begin{equation*}
\begin{aligned}
\Theta^{(f)}_{x}\circ\Theta^{(f)}_y(e)&=\Theta^{(f)}_x(e+y(h-1-y)f^{-1})\\
&=e+x(h-1-x)f^{-1}+y(h-2x-y-1)f^{-1}\\
&=e+(x+y)(h-1-x-y)f^{-1}\\
&=\Theta^{(f)}_{x+y}(e).
\end{aligned}
\end{equation*}
\qed
\end{myproof}
Now we can define Mathieu's twisting functors from \cite{Ma00} in our situation. Denote by $U$-Mod the category 
of all $U$-modules and their morphisms. Let $u\in\{ q, f\}$ and $x\in\C$. Then the \emph{twisting functor} 
\begin{equation*}
B^{(u)}_x:U\text{-Mod}\rightarrow U\text{-Mod}
\end{equation*}
is defined as composition of the following three functors:
\begin{enumerate}
\item[(i)]
the induction functor $\Ind:= U^{(u)}\otimes_U{}_-$,
\item[(ii)]
twisting the $U^{(u)}$-action by $\Theta^{(u)}_{x}$,
\item[(iii)]
the restriction functor $\Res$.
\end{enumerate}
More explicitly, for $V\in U$-Mod, then the module $B^{(u)}_x(V)$ is isomorphic to $U^{(u)}_x\otimes_U V$, where
$U^{(u)}_x$ is a $U$-bimodule obtained as follows: it coincides with the algebra $U^{(u)}$ as a vector space,
the right action is given by multiplication $p\cdot a:=pa$, and the left action is given by multiplication twisted by 
$\Theta^{(u)}_{x}$, that is $a\cdot p:=\Theta^{(u)}_{x}(a)p$ (here $p\in U^{(u)}$ and $a\in U$).
\begin{mylem}
\label{indres}
Let $V$ be a $U$-module on which $u\in\{q,f\}$ acts bijectively, and let $W$ be a $U^{(u)}$-module. Then 
\begin{enumerate}[$($a$)$]
\item
$\Ind\circ\Res(W)\cong W$.
\item
$\Res\circ\Ind(V)\cong V$.
\end{enumerate}
\end{mylem}
\begin{proof}
Since $u$ acts bijectively on both $V$ and $\Res(W)$, the linear map $v\mapsto 1\otimes v$ 
defines an isomorphism from $V$ or $W$ to $\Res\circ\Ind(V)$ or $\Ind\circ\Res(W)$, respectively.
\end{proof}
The next result facilitates the composition and inversion of twisting functors when applied to $\mathcal{N}$. 
\begin{myprop}
\label{Bcomp}
Let $x,y\in\C$, $u\in\{ q,f\}$. Then we have the following:
\begin{enumerate}[$($a$)$]
\item\label{Bcomp.1}
$B^{(u)}_x\circ B^{(u)}_y\cong B^{(u)}_{x+y}$.
\item\label{Bcomp.2}
If $V\in U$-Mod and $u$ acts bijectively on $V$, then $B^{(u)}_0(V)\cong V$.
\end{enumerate}
\end{myprop}
\begin{myproof}
By definition, for $V\in U$-Mod we have natural isomorphisms 
\begin{equation}
\begin{aligned}
B^{(u)}_x\circ B^{(u)}_y(V)&=\Res((\Ind(\Res((\Ind(V))_y)))_x)\\ &\cong \Res(((\Ind(V))_y)_x)\\ &\cong \Res((\Ind(V))_{x+y})\\ &=B^{(u)}_{x+y}(V),
\end{aligned}
\end{equation}
where the first isomorphism follows from Lemma~\ref{indres}, 
and the second from Proposition~\ref{thetaadditive}. This implies claim \eqref{Bcomp.1}.

Again by definition, we have 
\begin{equation}
\begin{aligned}
B^{(u)}_0(V)&=\Res((\Ind(V))_0)\\ &\cong\Res(\Ind(V))\\ &\cong V,
\end{aligned}
\end{equation}
where the second isomorphism follows from Lemma~\ref{indres}, proving claim \eqref{Bcomp.2}.
\qed
\end{myproof}
The modules $B^{(q)}_x(N(\lambda,c))$ are described explicitly in the following proposition.
\begin{myprop}
\label{twistedmodule}
For $\lambda\in -\frac{1}{2}+\N$, $x\in\C$ and $c\in\C^*$, the $U$-module ${B}^{(q)}_x(N(\lambda,c))$ has basis $\{ v_{i,j}\}_{i,j\in\mathbb{Z}, 0\le j\le\lambda+\frac{1}{2}}$, where $v_{i,j}=q^if^j\otimes_U v$, and $v$ is a  highest weight vector of $N(\lambda,c)$. The action of $U$ in this basis is given by:
\begin{equation}
\label{qqaction}
\begin{array}{rcl}
qv_{i,j}&=&v_{i+1,j};\\
fv_{i,j}&=&
\begin{cases}
v_{i,j+1},&\text{if $j<\lambda+\frac{1}{2}$};\\
-\sum_{s=0}^{\lambda+\frac{1}{2}}{\frac{1}{(2c)^{\lambda+\frac{3}{2}-s}}\binom{\lambda+\frac{3}{2}}{s}v_{i+2\lambda+3-2s,s}},&\text{if $j=\lambda+\frac{1}{2}$;}
\end{cases}
\\
zv_{i,j}&=&cv_{i,j};\\
hv_{i,j}&=&(\lambda-(i+x)-2j)v_{i,j};\\
pv_{i,j}&=&-jv_{i+1,j-1}+c(i+x)v_{i-1,j};\\
ev_{i,j}&=&j(\lambda+1-(i+x)-j)v_{i,j-1}+\frac{1}{2}c(i+x)((i+x)-1)v_{i-2,j}.
\end{array}
\end{equation}
\end{myprop}
\begin{myproof}
Using definitions, Proposition~\ref{pbwu} and the fact that $z, h, p$ and $e$ all act like scalars on $v$,
we get that the vector space $B^{(q)}_x(N(\lambda,c))$ is spanned by $\{ v_{i,j}\}_{i\in\Z, j\in\N}$. For $j>\lambda+\frac{1}{2}$, however, $v_{i,j}=q^if^{j-\lambda-\frac{3}{2}}(fv_{0,\lambda+\frac{1}{2}})$, which is a linear combination of elements $v_{i',j'}$ such that $0\le j'<j$. So, by induction, ${B}^{(q)}_x(N(\lambda,c))$ is spanned by $\{v_{i,j}\}_{i,j\in\Z, 0\le j\le \lambda+\frac{1}{2}}$ as well. 
\par
The elements of the latter set are linearly independent because if for some $m\in\N$ and $c_{-m},\dots,c_m\in\C$ we would have $\sum_{|i|<m, 0\le j\le \lambda+\frac{1}{2}}{c_i v_{i,j}}=0$, then it would follow that 
\begin{multline*}
0=\sum_{|i|<m, 0\le j\le \lambda+\frac{1}{2}}{c_iv_{i+m,j}}
=\sum_{|i|<m, 0\le j\le \lambda+\frac{1}{2}}{c_i(1\otimes q^{i+m}f^jv)}=\\
={1\otimes \sum_{|i|<m, 0\le j\le \lambda+\frac{1}{2}}c_iq^{i+m}f^jv}.
\end{multline*}
This is impossible as $v_{i+m,j}$ are linearly independent for
$|i|<m, 0\le j\le \lambda+\frac{1}{2}$.
\par
Let us determine the action of $U$ on the subspace of $B^{(q)}_0(N(\lambda,c))$ generated by $\{v_{i,j}\}_{i,j\in\N}$. This is clearly a submodule isomorphic to $N(\lambda,c)$ with the action given by \eqref{highestclassq}
(which coincides with evaluation of \eqref{qaction} at $x=0$). 
\par
On the subset $\{ v_{i,j}\}_{i,j\in\N}$ of $B^{(q)}_x(N(\lambda,c))$, an element $s\in U$ acts, by definition, as $\Theta^{(q)}_x(s)$ acts on $\{ v_{i,j}\}_{i,j\in\N}\subset \Ind(N(\lambda,c))$. This latter action is explicitly
described in Proposition~\ref{auto} and is easily seen to be given by \eqref{qaction}. The nontrivial cases 
$s\in\{ h, p, e\}$ are dealt with by direct calculations:
\begin{displaymath}
\Theta^{(q)}_x(h)v_{i,j}=(h-x)v_{i,j}=(\lambda-i-2j-x)v_{i,j}=(\lambda-(i+x)-2j)v_{i,j},
\end{displaymath}
\begin{displaymath}
\Theta^{(q)}_x(p)v_{i,j}=(p+xq^{-1}z)v_{i,j}=-jv_{i+1,j-1}+(ci+xc)v_{i-1,j}=-jv_{i+1,j-1}+c(i+x)v_{i-1,j},
\end{displaymath}
\begin{multline*}
\Theta^{(q)}_x(e)v_{i,j}=(e+xq^{-1}p+\frac{1}{2}x(x-1)q^{-2})v_{i,j}=\\=
\frac{1}{2}ci(i-1)v_{i-2,j}+j(\lambda+1-j-i)v_{i,j-1}+x(-jv_{i,j-1}+civ_{i-2,j})+\frac{1}{2}x(x-1)cv_{i-2,j}=\\=\frac{1}{2}c(i(i-1)+x(x-1)+2ix)v_{i-2,j}+j(\lambda+1-j-i-x)v_{i,j-1}=\\=\frac{1}{2}c(i+x)((i+x)-1)v_{i-2,j}+j(\lambda+1-j-(i+x))v_{i,j-1}.
\end{multline*}
Now, note that $sv_{i,j}=q^i\Theta^{(q)}_i(s)v_{0,j}$,
from which the action of $U$ on $v_{i,j}$ with $i<0$ is directly computable. Finally, the action of $U$ on $v_{i,j}$ for general $x\in\C$ and $i\in\Z$ can now be calculated similarly to when $i\ge 0$.
\qed
\end{myproof}


\section{Proof of main result}\label{s4}
The next few results will enable us to relate the modules in $\mathcal{N}$ to twisted highest weight modules.
\begin{mylem}
\label{inject}
Let $V$ be a simple weight $U$-module with finite-dimensional weight spaces. 
\begin{enumerate}[$($a$)$]
\item\label{inject.1}
If $e$ acts locally nilpotently on $V$, then $V$ is a highest weight module. If $f$ acts locally nilpotently on $V$, then $V$ is a lowest weight module.
\item\label{inject.2}
Assume, additionally, that $z$ acts like some $c\in\C^*$ on $V$. If $p$ acts locally nilpotently on $V$, then $V$ is a highest weight module. If $q$ acts locally nilpotently on $V$, then $V$ is a lowest weight module.
\end{enumerate}
\end{mylem}
\begin{proof}
We start with claim \eqref{inject.1}. We will consider this claim for the element $e$, the other part of the claim is similar. Assume that $V$ is not a highest weight module. Then $V$ is not a lowest weight module either, for
otherwise $e$ and $f$ would both act locally nilpotently on $V$ and hence $V$ would be a direct sum
of finite-dimensional modules when restricted to $\mathfrak{sl}_2$. As $V$ has finite-dimensional weight spaces,
this would mean that $V$ is finite-dimensional, hence highest weight, a contradiction.

Therefore $V$ is either a dense simple $\mathfrak{sl}_2$-module as in \cite[Section~3.3]{Ma10} or is in
$\mathcal{N}$. However, $e$ does not act locally nilpotently on simple dense $\mathfrak{sl}_2$-module, so
$V$ is in $\mathcal{N}$. By Theorem~\ref{chinesethm}, we have $\mathrm{supp}(V)=\lambda+\mathbb{Z}$
for some $\lambda\in\mathbb{C}$ and all nonzero weight spaces of $V$ have the same dimension.
By \cite[Lemma~3.3]{Ma00}, $V$ has finite length as an $\mathfrak{sl}_2$-module. The only simple weight
$\mathfrak{sl}_2$-modules on which $e$ acts locally nilpotently are highest weight modules,
therefore, as an $\mathfrak{sl}_2$-module, $V$ has a finite filtration with subquotients being
highest weight modules. Therefore $V$ must have a highest weight, a contradiction. This proves claim \eqref{inject.1}.

Now we prove claim \eqref{inject.2}. Again we prove it for the element $p$, the other part is similar.
Take any nonzero weight vector $v\in V$ such that $pv=0$. Let $\lambda$ be the weight of $v$.
By claim \eqref{inject.1} we may assume that the action of $e$ on $V$ is injective. Since $e$ and $p$ 
commute, we have that for every $i\in \N$ the element $v_i:=e^iv$ is nonzero and satisfies $pv_i=0$. 

Next we observe that the action of $q$ on $V$ is injective for otherwise it would be locally nilpotent
and then $q^kv=0$ for some $k$. The linear span of $v,qv,q^2v,\dots,q^{k-1}v$ is then a finite-dimensional
space stable under the action of both $q$ and $p$. As $[p,q]=z$ commutes with $p$, by the Kleinecke-Shirokov Theorem
it follows that the only eigenvalue of $z$ is zero, which contradicts our assumption on the action of $z$.

%
Finally we claim that $\{q^{2i}v_{i}\}_{i\in\N}$ is an infinite set of linearly independent elements. We use that $p^{2k}q^{2i}v_{i}=(\prod_{j=0}^{2k-1}(2i-j))c^{2k}q^{2(i-k)}v_{i}$ is zero if and only if $k>i$, where $k\in\N_{>0}$. From this we see that any two elements $q^{2i_1}v_{i_1}$ and $q^{2i_2}v_{i_2}$, where $i_1\ne i_2$, are distinct, hence follows the first claim. As for the second claim, assume it is false and let $\sum_{i=0}^{k}a_i q^{2i}v_{i}=0$ be a nontrivial relation of minimal possible $q$-degree $2k>0$ (so in particular $a_k\neq 0$). But applying $p^{2k}$ we obtain $(\prod_{j=0}^{2k-1}(2k-j))c^{2k}a_kv_{k}=0$, a contradiction.


As a result, we have infinitely many linearly independent elements of weight $\lambda$, a contradiction.
\end{proof}
\begin{mycor}
\label{biject}
Let $L\in\mathcal{N}$. Then both $e$ and $f$ act bijectively on $L$.
If in addition $z$ acts nonzero on $L$, then also both $p$ and $q$ act bijectively on $L$. 
\end{mycor}
\begin{myproof}
From Lemma~\ref{inject} it follows that the actions of $e$, $f$, $p$ and $q$ on $L$ are not locally
nilpotent. By Lemma~\ref{localnil},  these actions are thus injective. By Theorem~\ref{chinesethm},
these actions restrict to  injective actions between finite-dimensional vector spaces of the same 
dimension. Therefore they all are bijective.
\qed
\end{myproof}
\begin{myprop}
\label{highest2}
Let $V$ be a weight $U$-module such that $\supp(V)\subset\lambda+\mathbb{Z}$ for some $\lambda$
and $\sup_{i\in\mathbb{Z}}\dim V_{\lambda+i}<\infty$. Assume that there is some $0\ne v\in V$ 
such that $ev=0$ or $pv=0$. Then $V$ has a simple highest weight submodule.
\end{myprop}
\begin{proof}
By assumption,
\begin{equation*}
W_e:=\{ v\in V|e^nv=0 \text{ for some }n\in\N\}\ne 0
\end{equation*}
or
\begin{equation*}
W_p:=\{ v\in V|p^nv=0 \text{ for some }n\in\N\}\ne 0.
\end{equation*}
By Lemma~\ref{localnil}, $W_e$ and $W_p$ are submodules. Note that $V$ has finite length
(already as an $\mathfrak{sl}_2$-modules by \cite[Lemma~3.3]{Ma00}).
Let $W$ be a simple submodule of $W_e$ or $W_p$. By Lemma~\ref{inject}, $W$ is a highest weight module.
\end{proof}
\begin{mylem}
\label{noninject}
Let $L\in\mathcal{N}$. 
\begin{enumerate}[$($a$)$]
\item\label{noninject.1}
There is $x\in\mathbb{C}$ and $0\ne v\in {B}^{(f)}_x(L)$ such that $ev=0$. 
\item\label{noninject.2}
Assume that $z$ does not act like $0$ on $L$. Then there is $x\in\C$ and $0\ne v\in B^{(q)}_x(L)$ such that $pv=0$. 
\end{enumerate}
\end{mylem}
\begin{myproof}
Let us first show that, for a non-constant polynomial $g(x)$ and $n\times n$-matrices $A$ and $B$, which $B$ invertible, there is some $x\in\C$ such that $A+g(x)B$ is not invertible. It suffices to prove existence of an $x$ such that $\det(A+g(x)B)=0$. By Leibniz' determinant formula,
\begin{equation*}
\begin{aligned}
\det(A+g(x)B)&=\sum_{\sigma\in S_n}{\sgn(\sigma)\prod_i^n(A+g(x)B)_{i,\sigma_i}}\\
&=\sum_{\sigma\in S_n}\sgn(\sigma)\prod_i^n{(A_{i,\sigma_i}+g(x)B_{i,\sigma_i})}\\
&=g(x)^n\sum_{\sigma\in S_n}\sgn(\sigma)\prod_i^n{B_{i,\sigma_i}}+r(x)\\
&=\det(B)g(x)^n+r(x),
\end{aligned}
\end{equation*}
where $r(x)$ is some polynomial of degree strictly smaller than that of $g(x)^n$. Since $\det(B)\ne 0$, the claim
follows from the fundamental theorem of algebra.

To prove claim \eqref{noninject.1} it suffices to show there is $x\in\C$ such that $e$ does not act 
injectively on ${B}^{(f)}_x(L)$. By the definition of $B^{(f)}_x$ and \eqref{autoformulae2}, it is equivalent to show that for some $x\in\C$, the element $e+x(h-1-x)f^{-1}$ does not act injectively on $\Ind(L)$, which, as follows from Corollary~\ref{biject}, is isomorphic to $L$ as a  vector space. Let $\lambda\in \supp(L)$, and fix some bases in $L_\lambda$ and $L_{\lambda+2}$. By Theorem~\ref{chinesethm}, these spaces have the same (finite) dimension, say $n$, so $e\vert_{L_\lambda}$ and $f\vert_{L_{\lambda+2}}$ are given by $n\times n$-matrices, $E$ and $F$, respectively,
with $F$ invertible. Moreover, $h\vert_{L_{\lambda}}=\lambda$. It now suffices to show that for some $x\in \C$, the matrix $E+x(\lambda-1-x)F^{-1}$ is not invertible. This follows from the previous paragraph and
proves claim \eqref{noninject.1}.

Similarly, to prove \eqref{noninject.2} we have to show, under the assumption that $z$ acts like $c\ne 0$, that there is some $x\in\C$ such that the operator $\Theta^{(q)}_x(p)=p-xzq^{-1}$ does not act injectively on $\Ind(L)$. Let the actions $p\vert_{L_\lambda}$ and $q\vert_{L_{\lambda+1}}$ be given by some $n\times n$-matrices, $P$ and $Q$, respectively. 
Then $Q$ is invertible and we need to show that for some $x\in\C$ the matrix $P-xcQ^{-1}$ is not invertible. This 
follows again from the first part of the proof.
\qed
\end{myproof}
\begin{myprop}
\label{highest}
Let $L\in\mathcal{N}$. 
There are $x\in\C$, $c\in\C^*$ and $\lambda\in -\frac{1}{2}+\N$ such that $L\cong B^{(f)}_{-x}(N)$.
\end{myprop}
\begin{myproof}
For any $x\in\C$, one sees from Proposition~\ref{auto} that $B^{(f)}_x(L)$ is a weight module, the weight spaces of which have the same, finite, dimension as those of $L$ (all weight spaces of $L$ have the same dimension by Theorem~\ref{chinesethm}). By Lemma~\ref{noninject}, there is $x\in\C$ and $0\ne v\in B^{(f)}_x(L)$ such that $ev=0$. Then,
by Proposition~\ref{highest2}, there is a simple highest weight submodule $N$ of $B^{(f)}_x(L)$. Then the module $B^{(f)}_{-x}(N)$ is a submodule of $B^{(f)}_{-x}(B^{(f)}_{x}(L))$, which, by Proposition~\ref{Bcomp} and Corollary~\ref{biject}, is isomorphic to $L$. But $L$ is simple and hence $L\cong B^{(f)}_{-x}(N)$. 
\par
Also we know from Proposition~\ref{highestclass} that one of the following cases holds:
\begin{enumerate}
\item[(i)]
$pN=0=qN$.
\item[(ii)]
$N\cong M(\lambda,c)$ for some $\lambda\in\C\backslash (-\frac{1}{2}+\N)$ and $c\in\C^*$.
\item[(iii)]
$N\cong N(\lambda,c)$ for some $\lambda\in -\frac{1}{2}+\N$ and $c\in\C^*$.
\end{enumerate}
Now, in the first case, $p$ and $q$ would act on $L\cong B^{(f)}_{-x}(N)$ like $\Theta^{(f)}_{-x}(p)$ and $\Theta^{(f)}_{-x}(q)$, respectively, act on $\Ind(N)$. From Proposition~\ref{auto} we see that these elements would act like zero, which  contradicts $L\in\mathcal{N}$. In the second case, the dimensions of the weight spaces in
$N$ are unbounded, so the same is true for $B^{(f)}_{-x}(N)\cong L$, contradicting $L\in \mathcal{N}$ because of Theorem~\ref{chinesethm}. Therefore the third alternative must hold, so that $L\cong B^{(f)}_{-x}(N(\lambda,c))$.
\qed
\end{myproof}
\begin{mycor}
\label{znonzero}
On an arbitrary $L\in\mathcal{N}$, $z$ does not act like $0$.
\end{mycor}
\begin{proof}
By Proposition~\ref{highest}, we have $L\cong B^{(f)}_{-x}(N(\lambda,c))$ for certain $x,c,\lambda\in\C$ where $c\ne 0$ (so that $z$ does not act like $0$ on $N(\lambda,c)$). The action of $z$ is not affected by $B^{(f)}_{-x}$ which yields the claim.
\end{proof}

\begin{myproofof}
Let $L\in \mathcal{N}$. By Corollary~\ref{znonzero}, we know that the action of $z$ on $L$ is nonzero. Now we can apply the same argument as in the proof of Proposition~\ref{highest}, but with $q$ instead of $f$ and $p$ instead of $e$, and with case (i) now ruled out because otherwise $z$ would act like $0$ on $N$, and therefore on $B^{(q)}_{-x}(N)$ as well. This gives that $L\cong B^{(q)}_{-x}(N(\lambda,c))$, for some $x\in\C$, $c\in\C^*$ and $\lambda\in -\frac{1}{2}+\N$.
\par
Conversely, let $\lambda\in -\frac{1}{2}+\N$, $c\in\C^*$ and $x\in\C$ be arbitrary. Let also $S\subseteq B^{(q)}_x(N(\lambda,c))$ be a simple (not necessarily proper) submodule. 
Then, by Proposition~\ref{Bcomp}, $B_{-x}^{(q)}(S)\subseteq B_0^{(q)}(N(\lambda,c))$ is a submodule as well, and is nonzero because of Lemma~\ref{indres}. In addition, $N(\lambda,c)$ is a submodule of $B_0^{(q)}(N(\lambda,c))$. Then $B_{-x}^{(q)}(S)\cap N(\lambda,c)\subseteq N(\lambda,c)$ is yet another submodule. This submodule is nonzero, because the action of $q$ is injective on $B_0^{(q)}(N(\lambda,c))$ and thus $B_{-x}^{(q)}(S)$ has vectors of arbitrarily low weight, some of which must then lie in $N(\lambda,c)$ as well. From simplicity of $N(\lambda,c)$ it follows that $B_{-x}^{(q)}(S)\cap N(\lambda,c)=N(\lambda,c)$ and therefore $B_{0}^{(q)}(S)=B_x^{(q)}(N(\lambda,c))$. Then either $S=B_x^{(q)}(N(\lambda,c))$, in which case $B_x^{(q)}(N(\lambda,c))$ is simple, or $S$ is a highest weight module with maximal dimension of weight spaces equal to that of $N(\lambda,c)$. From Proposition~\ref{highestclass} we in this latter case get that in fact $S\cong N(\lambda,c)$, so that $B_{0}^{(q)}(S)\cong B_0^{(q)}(N(\lambda,c))$. All in all, we see that either $B^{(q)}_x(N(\lambda,c))$ is simple, or $B_{0}^{(q)}(N(\lambda,c))\cong B_x^{(q)}(N(\lambda,c))$. The latter case is investigated together with possible redundancy as follows.
\par
Assume that $B^{(q)}_{x_1}(N(\lambda,c))\cong B^{(q)}_{x_2}(N(\lambda',c'))$. That $c'=c$ is obvious. The isomorphism also implies equality of weight space dimensions, so it follows from Proposition~\ref{twistedmodule} that $\lambda'=\lambda$, and we furthermore get, by Proposition~\ref{Bcomp}, that $B^{(q)}_{x}(N(\lambda,c))\cong B^{(q)}_{0}(N(\lambda,c))$, where $x=x_1-x_2$. Then 
\begin{equation}
\mathbb{Z}+\lambda-x=\supp(B^{(q)}_{x}(N(\lambda,c)))= \supp(B^{(q)}_{0}(N(\lambda,c)))=\mathbb{Z}+\lambda
\end{equation}
so that $x\in \mathbb{Z}$. 
\par
Whenever $x\in\mathbb{Z}$, one conversely easily sees from formulae~\ref{qaction} that
\begin{align*}
\Phi:B^{(q)}_0(N(\lambda,c))&\rightarrow {B}^{(q)}_x(N(\lambda,c))\\
v_{i+x,j}&\mapsto v_{i,j}
\end{align*}
defines an isomorphism. 
\par
Thus $B^{(q)}_0(N(\lambda,c))$ and ${B}^{(q)}_x(N(\lambda,c))$ are isomorphic if and only if $x\in\mathbb{Z}$, so that in total (via another application of Proposition~\ref{Bcomp}) we have that $B^{(q)}_{x_1}(N(\lambda,c))$ is isomorphic to $B^{(q)}_{x_2}(N(\lambda',c'))$ if and only if $\lambda=\lambda'$, $c'=c$ and $x_1-x_2\in\Z$. This completes the proof.
\qed
\end{myproofof}
\vspace{5mm}

\noindent
{\bf Acknowledgements.}
This paper is an adaptation of author's Master Thesis. The author thanks his advisor, Volodymyr Mazorchuk, 
for formulating the problem, stimulating discussions and help in preparation of this paper.

\section*{References}





\bibliographystyle{elsarticle-num}
\bibliography{refs}







\end{document}